\documentclass[12pt, reqno]{amsart}
 \usepackage{amsmath, amsthm, amscd, amsfonts, amssymb, graphicx, color, float}
\usepackage[bookmarksnumbered, colorlinks, plainpages]{hyperref}

\newtheorem{theorem}{Theorem}[section]
\newtheorem{lemma}[theorem]{Lemma}

\newtheorem{corollary}[theorem]{Corollary}
\theoremstyle{definition}

\theoremstyle{remark}
\newtheorem{remark}[theorem]{Remark}
\numberwithin{equation}{section}

\begin{document}

\title[Inequalities in noncommutative probability spaces]{Inequalities for sums of random variables in noncommutative probability spaces}
\author{Ghadir Sadeghi}
\address{$^1$ Department of Mathematics and Computer
Sciences, Hakim Sabzevari University, P.O. Box 397, Sabzevar, Iran}
\email{ghadir54@gmail.com, g.sadeghi@hsu.ac.ir}

\author{Mohammad Sal Moslehian}
\address{$^2$ Department of Pure Mathematics, Center of Excellence in
Analysis on Algebraic Structures (CEAAS), Ferdowsi University of
Mashhad, P. O. Box 1159, Mashhad 91775, Iran}
\email{moslehian@um.ac.ir, moslehian@member.ams.org}


\begin{abstract}
In this paper, we establish an extension of a noncommutative Bennett inequality with a parameter $1\leq r\leq2$ and use it together with some noncommutative techniques to establish a Rosenthal inequality. We also present a noncommutative Hoeffding inequality as follows: Let $(\mathfrak{M}, \tau)$ be a noncommutative probability space, $\mathfrak{N}$ be a von Neumann subalgebra of $\mathfrak{M}$ with the corresponding conditional expectation $\mathcal{E}_{\mathfrak{N}}$ and let subalgebras $\mathfrak{N}\subseteq\mathfrak{A}_j\subseteq\mathfrak{M}\,\,(j=1, \cdots, n)$ be successively independent over $\mathfrak{N}$. Let $x_j\in\mathfrak{A}_j$ be self-adjoint such that $a_j\leq x_j\leq b_j$ for some real numbers $a_j<b_j$ and $\mathcal{E}_{\mathfrak{N}}(x_j)=\mu$ for some $\mu\geq 0$ and all $1\leq j\leq n$. Then for any $t>o$ it holds that
\begin{eqnarray*}
{\rm Prob}\left(\left|\sum_{j=1}^n x_j-n\mu\right|\geq t\right)\leq 2 \exp\left\{\frac{-2t^2}{\sum_{j=1}^n(b_j-a_j)^2}\right\}.
\end{eqnarray*}
\end{abstract}

\maketitle

\section{Introduction}

By a noncommutative probability space $(\mathfrak{M}, \tau)$ we mean a von Neumann algebra $\mathfrak{M}$ on a Hilbert
space $\mathfrak{H}$ with unit element $1$ equipped with a faithful normal finite trace $\tau$ such that $\tau(1)=1$. The modules $|x|$ of $x\in \mathfrak{M}$ is defined by continuous functional calculus as $|x|=(x^*x)^{1/2}$.

For each self-adjoint operator $x\in \mathfrak{M}$, there exists a unique spectral measure $E$ as a $\sigma$-additive mapping with respect to the strong operator topology from the Borel $\sigma$-algebra $\mathcal{B}(\mathbb{R})$ of $\mathbb{R}$ into the set of all orthogonal projections such that for every Borel function $f: \sigma(x)\to \mathbb{C}$ the operator $f(x)$ is defined by $f(x)=\int f(\lambda)dE(\lambda)$, in particular, $x=\int \lambda dE(\lambda)$ and $\chi_B(x)=\int_BdE(\lambda)=E(B)$. In addition,
\begin{eqnarray}\label{tc}
\tau\left(\chi_{[t,\infty)}\left(\left|x\right|\right)\right)
=\tau\left(\chi_{[t,\infty)}\left(x\right)\right)+\tau\left(\chi_{[t,\infty)}\left(-x\right)\right)\,.
\end{eqnarray}
Further, if $x \geq 0$ and $t>0$, then $\chi_{[t,\infty)}(x)t \leq x$. Hence we get the inequality
\begin{eqnarray*}
\tau(\chi_{[t,\infty)}(x))\leq t^{-1}\tau(x)\,,
\end{eqnarray*}
which is called the Markov inequality in the literature. For a self-adjoint element $x\in \mathfrak{M}$, it follows from the Markov inequality that $$\tau(\chi_{[t,\infty)}(x))=\tau(\chi_{[e^t,\infty)}(e^x))\leq e^{-t}\tau(e^x),$$
fromwhere we reach the exponential Chebyshev inequality as follows:
\begin{eqnarray}\label{GH1}
\tau(\chi_{[t,\infty)}(x))\leq e^{-t}\tau(e^x)\,.
\end{eqnarray}
As in the commutative context, we use the notation ${\rm Prob}(x\geq t):=\tau(\chi_{[t,\infty)}(x))$.

For any Borel set $A\subseteq\mathbb{R}$, we define $\nu(A)=\tau(E(A))$. Then $\nu$ is a scalar-valued spectral measure for $x$ and $\nu(\mathbb{R})=1$. In addition,
\begin{align}\label{MOS1}
\tau(f(x))=\tau\left(\int f(\lambda)dE(\lambda)\right)=\int f(\lambda)d\tau\left(E(\lambda)\right)=\int f(\lambda)d\nu(\lambda)\,.
\end{align}
By the measurable functional calculus \cite{CO} there is a $*$-homomorphism $\pi:L^{\infty}(\nu)\to\mathfrak{M}$ depending on $x$ such that $\pi(f)=f(x)$ for all $f\in L^{\infty}(\nu)$.

For $1\leq p<\infty$, $L^p(\mathfrak{M},\tau)$ is defined as the completion of $\mathfrak{M}$ with respect to the norm $\|x\|_p=\tau(|x|^p)^{\frac{1}{p}}$. Important special cases of these noncommutative spaces are the usual $L^p$-spaces and the Schatten $p$-classes. For further information we refer the reader to \cite{Y, P, MTS, SAD}.
Let $x\in\mathfrak{M}$ be positive. For $p\geq 1$ and positive $x\in\mathfrak{M}$, from \eqref{MOS1}, we have
\begin{eqnarray*}
\|x\|_p^p&=& \tau(x^p)=\int_0^\infty \lambda ^pd\nu(\lambda)=\int_0^{\infty} pt^{p-1}\nu([t,\infty))dt\\&=&\int_0^{\infty} pt^{p-1}\tau(\chi_{[t,\infty)}(x))dt.
\end{eqnarray*}

Let $\mathcal{P}$ be the lattice of projections of $\mathfrak{M}$. Set $p^{\perp}=1-p$ for $p\in \mathcal{P}$. Given a family of projections $(p_{\lambda})_{\lambda\in\Lambda}\subseteq\mathcal{P}$, we denote by $\vee_{\lambda\in\Lambda}p_{\lambda}$ (resp., $\wedge_{\lambda\in\Lambda}p_{\lambda}$) the projection from $\mathfrak{H}$ onto the closed subspace generated by $p_{\lambda}(\mathfrak{H})$ (resp. onto the subspace $\cap_{\lambda\in\Lambda}p(\mathfrak{H})$). Consequently, $\left(\vee_{\lambda\in\Lambda}p_{\lambda}\right)^{\perp}=\wedge_{\lambda\in\Lambda}p_{\lambda}^{\perp}$. Two projections $p$ and $q$ are said to be equivalent if there exists a partial isometry $u\in\mathfrak{M}$ such that $u^*u=p$ and $uu^*=q$. In this case, we
write $p\sim q$. If $p$ is equivalent to a projection $q_1 \leq q$, we write $p\prec q$. We need the following elementary properties of projections (see \cite{N}).
\begin{lemma}\label{NE}
Let $p$ and $q$ be two projection of $\mathfrak{M}$. Then\\
(i) $p\vee q-q\sim p-p\wedge q$.\\
(ii) If $p\wedge q=0$ then $p\prec q^{\perp}$\\
(iii) If $p$ and $q$ are equivalent projections in $\mathfrak{M}$ then $\tau(p)=\tau(q).$\\
(iv) If $(p_{\lambda})_{\lambda\in\Lambda}$ is a family of projections in $\mathfrak{M}$ then $\tau\left(\vee_{\lambda\in\Lambda}p_{\lambda}\right)\leq\sum_{\lambda\in\Lambda}\tau(p_{\lambda})$.
\end{lemma}
Let $\mathfrak{N}$ be a von Neumann subalgebra of $\mathfrak{M}$. Then there exists a map $\mathcal{E}_{\mathfrak{N}}:\mathfrak{M}\to\mathfrak{N}$ satisfying the following properties:\\
(i) $\mathcal{E}_{\mathfrak{N}}$ is a normal contraction positive map projecting $\mathfrak{M}$ onto $\mathfrak{N}$;\\
(ii) $\mathcal{E}_{\mathfrak{N}}(axb) =a\mathcal{E}_{\mathfrak{N}}(x)b$ for any $x\in\mathfrak{M}$ and $a, b\in\mathfrak{N}$;\\
(iii) $\tau\circ\mathcal{E}_{\mathfrak{N}}=\tau$.\\
Moreover, $\mathcal{E}_{\mathfrak{N}}$ is the unique map satisfying (i) and (ii). The map $\mathcal{E}_{\mathfrak{N}}$ is called the conditional expectation of $\mathfrak{M}$ with respect to $\mathfrak{N}$. We say that two subalgebras $\mathfrak{N}\subseteq \mathfrak{A}, \mathfrak{B}\subseteq\mathfrak{M}$ are independent over $\mathfrak{N}$ if $\mathcal{E}_{\mathfrak{N}}(xy)=\mathcal{E}_{\mathfrak{N}}(x)\mathcal{E}_{\mathfrak{N}}(y)$ for all $x\in\mathfrak{A}, y\in\mathfrak{B}$. In particular, two random variable $X$ and $Y$ of a commutative von Neumann algebra $L^\infty(\mu)$ in which $\mu$ is a probability measure are independent if the algebras they generate are independent over the complex field $\mathbb{C}$. A sequence of subalgebras
$\mathfrak{N}\subseteq\mathfrak{A}_1,\ldots,\mathfrak{A}_n\subseteq\mathfrak{M}$ is
called successively independent over $\mathfrak{N}$ if $\mathfrak{A}_{k+1}$ is independent of the algebra $\mathfrak{M}(k)$ generated by $\mathfrak{A}_1,\ldots,\mathfrak{A}_k$. For further information the reader is refereed to \cite{Xu1, Xu2}.

In 1970, Rosenthal \cite{ROS} presented an inequality to describe isomorphic types of some subspaces in $L_p$-spaces. It indeed gives a bound for the $p$-norm of independent mean $0$ random variables. More precisely, it says that for any $p\geq 2$ there exists a constant $c(p)$ such that for any $n\in \mathbb{N}$ and any independent mean $0$ random variables $f_1, \cdots, f_n$ it holds that
\begin{equation}\label{MOS0}
{\mathbb E}\Biggl|\sum_{k=1}^n
f_k\Biggr|^p \le c(p) \Biggl( \Biggl( \sum_{k=1}^n {\mathbb E}|f_k|^2
\Biggr)^{p/2}+ \sum_{k=1}^n
{\mathbb E}|f_k|^p \Biggr)\,.
\end{equation}\
Burkholder \cite{BUR} generalized Rosenthal's inequality in the context of martingales. Since then this inequality has been generalized and applied by many mathematicians; see, e.g., \cite{JSZ, JX} and references therein. Recently, Junge and Zeng \cite{JZ} extended the Bennett and Bernstein inequalities to the noncommutative setting and derive a version of the Rosenthal inequality from Bernstein's inequality by using the properties of Gamma function. In probability theory, the Bennett inequality (Bernstein inequality, resp.) gives an upper bound on the probability that the sum of independent random variables deviates from its expected value (deviates from its mean, resp.) by more than a fixed amount, see \cite{BEN}.

In this paper, we establish an extension of the noncommutative Bennett inequality due to Junge and Zeng \cite{JZ} and use it together with some noncommutative techniques to prove the Rosenthal inequality with a parameter $1\leq r\leq2$. We also prove a noncommutative Hoeffding inequality. The Hoeffding inequality \cite{HOE} gives a probability bound for the deviation between the average of $n$ independent bounded random variables and its mean (see Corollary \ref{cor1}). There have been several generalizations and applications of this significant inequality, see \cite{1, 3}.


\section{Bennet inequality}

We provide an improved version of the noncommutative Bennett inequality based on the arguments of \cite[Theorem 01]{JZ}.
\begin{theorem} (Noncommutative Bennett inequality)\label{B}
Let $\mathfrak{N}\subseteq\mathfrak{A}_j\subseteq\mathfrak{M}$ be successively independent over $\mathfrak{N}$ and $x_j\in\mathfrak{A}_j$ be self-adjoint and $1\leq r\leq2$ such that
\begin{itemize}
\item $\mathcal{E}_{\mathfrak{N}}(x_j)\leq0$,

\item $\mathcal{E}_{\mathfrak{N}}(|x_j|^r)\leq b_j^r$,

\item $\|x_j\|\leq M$,

\end{itemize}
for some $M>0$ and all $1\leq j\leq n$. Then for each $t\geq0$,
\begin{align}\label{MOS2}
{\rm Prob}\left(\sum_{j=1}^nx_j\geq t\right)=\tau\left(\chi_{[t,\infty)}\left(\sum_{j=1}^nx_j\right)\right)\leq \exp\left(\frac{-b}{M^r}\Phi\left(\frac{tM^{r-1}}{b}\right)\right),
\end{align}
where $\Phi(\alpha)=(1+\alpha)\log(1+\alpha)-\alpha$ and $b=\sum_{j=1}^nb_j^r$.
\end{theorem}
\begin{proof}
Let $\lambda\geq0$. We have
\begin{eqnarray*}
\mathcal{E}_{\mathfrak{N}}\left(e^{\lambda x_n}\right)&=&\mathcal{E}_{\mathfrak{N}}\left(\sum_{k=0}^{\infty}\frac{(\lambda x_n)^k}{k!}\right)
=\sum_{k=0}^{\infty}\frac{\lambda^k}{k!}\mathcal{E}_{\mathfrak{N}}(x_n^k)\\
&=&1+\mathcal{E}_{\mathfrak{N}}(x_n)+\sum_{k=2}^{\infty}\frac{\lambda^k}{k!}\mathcal{E}_{\mathfrak{N}}(x_n^k)\\
&\leq&1+\sum_{k=2}^{\infty}\frac{\lambda^k}{k!}\mathcal{E}_{\mathfrak{N}}(x_n^k)\qquad\qquad\qquad\quad ({\rm by~} \mathcal{E}_{\mathfrak{N}}(x_n)\leq 0)\\
&\leq&1+\sum_{k=2}^{\infty}\frac{\lambda^k}{k!}\mathcal{E}_{\mathfrak{N}}(|x_n|^k)\qquad\qquad\qquad ({\rm by~} x_n^k\leq|x_n|^k)\\
&=&1+\sum_{k=2}^{\infty}\frac{\lambda^k}{k!}\mathcal{E}_{\mathfrak{N}}(|x_n|^r|x_n|^{k-r})\\
&\leq&1+\sum_{k=2}^{\infty}\frac{\lambda^k}{k!}\|x_n\|^{k-r}\mathcal{E}_{\mathfrak{N}}(|x_n|^r)\\
&&\qquad\qquad\qquad\qquad({\rm by~} |x_n|^r|x_n|^{k-r}\leq \|x_n\|^{k-r}|x_n|^r)\\
&\leq&1+\sum_{k=2}^{\infty}\frac{\lambda^k}{k!}\|x_n\|^{k-r}b_n^r\\
&=&1+\frac{b_n^r}{\|x_n\|^{r}}(e^{\lambda \|x_n\|}-1-\lambda \|x_n\|)\\
&\leq&\exp\left(\frac{b_n^r}{\|x_n\|^{r}}(e^{\lambda \|x_n\|}-1-\lambda \|x_n\|)\right).
\end{eqnarray*}
Note that the function $f(s):=\exp\left(\frac{e^{\lambda s}-1-\lambda s}{s^r}\right)$ is increasing for $s> 0$. It follows that
\begin{eqnarray}\label{GH2}
\mathcal{E}_{\mathfrak{N}}\left(e^{\lambda x_n}\right)\leq \exp\left(\frac{b_n^r}{M^{r}}(e^{\lambda M}-1-\lambda M)\right)\,.
\end{eqnarray}

It follows from \eqref{GH1} that
\begin{align}\label{GH3}
\tau\left(\chi_{[t,\infty)}\left(\sum_{j=1}^nx_j\right)\right)&=\tau\left(\chi_{[\lambda t,\infty)}\left(\sum_{j=1}^n\lambda x_j\right)\right)\nonumber\\
&\leq \exp(-\lambda t)\tau\left(e^{\sum_{j=1}^n\lambda x_j}\right).
\end{align}
Recall that the Golden--Thompson inequality \cite{RUS} stating that for all self-adjoint elements $z_1, z_2 \in \mathfrak{M}$,
\begin{eqnarray}\label{T1}
\tau(e^{z_1+z_2})\leq \tau(e^{z_1} e^{z_2})\,.
\end{eqnarray}
Hence
\begin{eqnarray*}
\tau\left(e^{\sum_{j=1}^n\lambda x_j}\right)&\leq&\tau\left(e^{\sum_{j=1}^{n-1}\lambda x_j}e^{\lambda x_n}\right) \qquad\qquad\qquad\qquad\qquad\quad({\rm by~ \eqref{T1}})\\
&=&\tau\left(\mathcal{E}_{\mathfrak{N}}\left(e^{\sum_{j=1}^{n-1}\lambda x_j}e^{\lambda x_n}\right)\right)\\
&=&\tau\left(\mathcal{E}_{\mathfrak{N}}\left(e^{\sum_{j=1}^{n-1}\lambda x_j}\right)\mathcal{E}_{\mathfrak{N}}\left(e^{\lambda x_n}\right)\right)\\
&\leq& \exp\left(\frac{b_n^r}{M^{r}}(e^{\lambda M}-1-\lambda M)\right)\tau\left(\mathcal{E}_{\mathfrak{N}}\left(e^{\sum_{j=1}^{n-1}\lambda x_j}\right)\right)\\
&&\qquad\qquad\qquad\qquad\qquad\quad({\rm by~ \eqref{GH2} ~and~traciality~of~\tau})\\
&\leq&\exp\left(\frac{\sum_{j=1}^nb_j^r}{M^{r}}(e^{\lambda M}-1-\lambda M)\right)\\
&&\qquad \qquad \qquad \qquad \qquad \qquad {\rm (by~iterating~} n-2 {\rm ~times)}
\end{eqnarray*}
We infer from the latter inequality together with \eqref{GH3} that
\begin{eqnarray*}
\tau\left(\chi_{[t,\infty)}\left(\sum_{j=1}^nx_j\right)\right)\leq\exp\left(-\lambda t+\frac{\sum_{j=1}^nb_j^r}{M^{r}}(e^{\lambda M}-1-\lambda M)\right).
\end{eqnarray*}
By basic calculus method we find the minimizing value
\begin{eqnarray*}
\lambda=\frac{1}{M}\log\left(1+\frac{tM^{r-1}}{\sum_{j=1}^nb_j^r}\right),
\end{eqnarray*}
which yields \eqref{MOS2}.
\end{proof}
Since both inequalities $\Phi(t)\geq \frac{t^2}{2+2t/3}$ and $\Phi(t)\geq\frac{t}{2}{\rm arcsinh}(\frac{t}{2})$ are valid for all $t\geq0$, so one can get the Bernstein and Prohorov inequalities from the Bennett's inequality as follows.
\begin{corollary}
Under the same hypothesis of Theorem \ref{B},
\begin{eqnarray}\label{BER}
\tau\left(\chi_{[t,\infty)}\left(\sum_{j=1}^nx_j\right)\right)\leq \exp\left(-\frac{t^2M^{r-2}}{2b+(2/3)tM^{r-1}}\right)
\end{eqnarray}
and
\begin{eqnarray*}
\tau\left(\chi_{[t,\infty)}\left(\sum_{j=1}^nx_j\right)\right)\leq \exp\left(-\frac{t}{2M}{\rm arcsinh}\left(\frac{tM}{2b}\right)\right)
\end{eqnarray*}
\end{corollary}
We can immediately deduce the following commutative Bernstein's Inequality.
\begin{corollary}
Let $X_1, \cdots, X_n$ be independent Bernoulli random variables taking values $1$ and $-1$ with probability $1/2$. Then
$${\rm Prob}\left(\left|\frac{1}{n}\sum_{i=1}^nX_i\right|\geq \varepsilon\right) \leq 2e^{\frac{-n\varepsilon^2}{2+(2\varepsilon/3)}}\,.$$
\end{corollary}

\section{Hoeffding inequality}

In this section we provide a noncommutative version of Hoeffding's inequality and present some consequences.
\begin{theorem} (Noncommutative Hoeffding inequality)\label{H}
Let $\mathfrak{N}\subseteq\mathfrak{A}_j\subseteq\mathfrak{M}$ be successively independent over $\mathfrak{N}$ and let $x_j\in\mathfrak{A}_j$ be self-adjoint such that $a_j\leq x_j\leq b_j$ for some real numbers $a_j<b_j$ and $\mathcal{E}_{\mathfrak{N}}(x_j)=\mu$ for some $\mu\geq 0$ and all $1\leq j\leq n$. Then
\begin{eqnarray}\label{MOS222}
{\rm Prob}\left(\left|S_n-n\mu\right|\geq t\right)\leq 2 \exp\left\{\frac{-2t^2}{\sum_{j=1}^n(b_j-a_j)^2}\right\}.
\end{eqnarray}
for any $t>0$, where $S_n=\sum_{j=1}^n x_j$.
\end{theorem}
\begin{proof}
First we show that if $x\in\mathfrak{A}_j$ is self-adjoint such that $a \leq x\leq b$ and $\mathcal{E}_{\mathfrak{N}}(x)=0$, then
\begin{eqnarray}\label{H1}
\mathcal{E}_{\mathfrak{N}}(e^{sx})\leq \exp\left\{\frac{s^2(b-a)^2}{8}\right\}
\end{eqnarray}
for any $s>0$.\\
Let $s>0$. Note that $t\mapsto e^{ts}$ is convex, therefore for any $a\leq \alpha\leq b$,
\begin{eqnarray*}
e^{s\alpha}\leq e^{sb}\frac{\alpha-a}{b-a}+e^{sa}\frac{b-\alpha}{b-a}\,.
\end{eqnarray*}
By the functional calculus we have
\begin{eqnarray*}
e^{sx}\leq e^{sb}\frac{x-a}{b-a}+e^{sa}\frac{b-x}{b-a}.
\end{eqnarray*}
Since $\mathcal{E}_{\mathfrak{N}}$ is a positive map and $\mathcal{E}_{\mathfrak{N}}(x)=0$, we reach
\begin{eqnarray*}
\mathcal{E}_{\mathfrak{N}}(e^{sx})\leq \frac{-a}{b-a}e^{sb}+\frac{b}{b-a}e^{sa}=e^{h(\alpha)},
\end{eqnarray*}
where $\alpha=s(b-a)$, $h(\alpha)=-\gamma \alpha+\log(1-\gamma+\gamma e^{\alpha})$ and $\gamma=-a/(b-a)$. In addition, $h(0)=h'(0)=0$ and $h''(\alpha)\leq \frac{1}{4}$ for all $\alpha>0$. By Taylor's Theorem there exists a real number $\xi\in(0,\alpha)$ such that
$$h(\alpha)=h(0)+\alpha h'(0)+\frac{\alpha^2}{2}h''(\xi)\leq\frac{\alpha^2}{8}=\frac{s^2(b-a)^2}{8}\,.$$
Hence
\begin{eqnarray*}
\mathcal{E}_{\mathfrak{N}}(e^{sx})\leq \exp\left\{\frac{s^2(b-a)^2}{8}\right\}\,.
\end{eqnarray*}

Second, for arbitrary value of $\mathcal{E}_{\mathfrak{N}}(x)$, setting $y:=x-\mu$ we get $a-\mu \leq y\leq b-\mu$ and $\mathcal{E}_{\mathfrak{N}}(y)=0$. Employing \eqref{H1} we reach
\begin{eqnarray}\label{H2}
\mathcal{E}_{\mathfrak{N}}\left(e^{s(x-\mu)}\right) \leq \exp\left\{\frac{s^2(b-a)^2}{8}\right\}\,.
\end{eqnarray}

Next, by the same reasoning as in the proof of Theorem \ref{B}, we get from \eqref{H2} that
\begin{eqnarray*}
\tau\left(e^{\lambda\sum_{j=1}^n(x_j-\mu)}\right)&\leq& \tau\left(\mathcal{E}_{\mathfrak{N}}\left(e^{\lambda\sum_{j=1}^{n-1}(x_j-\mu)}\right)\mathcal{E}_{\mathfrak{N}}\left(e^{\lambda (x_n-\mu)}\right)\right)\\
&\leq& e^{\frac{\lambda ^2(b_n-a_n)^2}{8}}\tau\left(\mathcal{E}_{\mathfrak{N}}\left(e^{\lambda\sum_{j=1}^{n-1}(x_j-\mu)}\right)\right)\leq \cdots\\
&\leq& \exp\left\{\frac{\lambda^2\sum_{j=1}^n(b_j-a_j)^2}{8}\right\}\,.
\end{eqnarray*}
Therefore \eqref{tc} and the exponential Chebyshev inequality \eqref{GH1} yield that
\begin{eqnarray*}
{\rm Prob}(\left|S_n-n\mu\right|\geq t)&=&2{\rm Prob}(S_n- n\mu\geq t)\\
&\leq& 2e^{-\lambda t}\exp\left\{\frac{\lambda^2\sum_{j=1}^n(b_j-a_j)^2}{8}\right\}.
\end{eqnarray*}
This is minimized when $\lambda=\frac{4t}{\sum_{j=1}^n(b_j-a_j)^2}$. Thus
\begin{eqnarray*}
{\rm Prob}(\left|S_n-n\mu\right|\leq 2\exp\left\{\frac{-2t^2}{\sum_{j=1}^n(b_j-a_j)^2}\right\}\,,
\end{eqnarray*}
which is the desired inequality.
\end{proof}
\begin{remark}
Under the same hypothesis of Theorem \ref{H} we can show that the bound in the Hoeffding inequality \eqref{MOS222} is sharper than that in the Bernstein inequality \eqref{BER} for $r=2$. Let us assume that We may assume that $\mathcal{E}_{\mathfrak{N}}(x_j)=0$ and $-1 \leq x_j \leq 1$. By the functional calculus $ |x_j|\leq 1$ so
$\mathcal{E}_{\mathfrak{N}}(|x_j|^2)\leq 1$ and $\|x_j\|\leq 1$. So $b=n$ and $M=1$ in the notation of Theorem \ref{B}. Then Heoffding inequality gives rise to
\begin{eqnarray*}
{\rm Prob}\left(\sum_{j=1}^nx_j\geq t\right) \leq \exp\left\{\frac{-t^2}{2n}\right\}\,.
\end{eqnarray*}
and from Bernstein inequality we have
\begin{eqnarray*}
{\rm Prob}\left(\sum_{j=1}^nx_j\geq t\right) \leq \exp\left\{\frac{-t^2}{2n+(2t/3)}\right\}\,.
\end{eqnarray*}
\end{remark}
The next result is the classical (commutative) version of the Hoeffding inequality.
\begin{corollary}[Hoeffding's Inequality \cite{HOE}]\label{cor1}
Let $a \leq X_1, \cdots, X_n\leq b$ be independent random variables with the expectation $\mathbb{E}(X_i)=\mu\,\,(i=1, \cdots, n)$. If $\overline{X}_n=\left(\sum_{i=1}^nX_i\right)/n$, then
$${\rm Prob}(\left|\overline{X}_n-\mu\right|\geq t) \leq 2e^{-2nt^2/(b-a)^2}\,.$$
\end{corollary}
In the special case, we immediately get
\begin{corollary}\label{cor1}
Let $0 \leq X_1, \cdots, X_n\leq 1$ be independent random variables with common mean $\mu$. Then with probability at least $1-\varepsilon$,
$$\left|\frac{1}{n}\sum_{i=1}^nX_i-\mu\right| \leq \sqrt{\frac{\log(2/\varepsilon)}{2n}}\,.$$
\end{corollary}

\section{Rosenthal inequality}

In this section we intend to prove a noncommutative Rosenthal inequality by using our noncommutative Bennet inequality. Our argument seems to be simpler than that of \cite{HU} for the case of usual random variables. We should notify that there is a refinement of it in the literature in which various approach is used, see \cite[Theorem 0.4]{JZ} and \cite{DIR}.
\begin{theorem}
Let $1\leq r\leq2\leq p<\infty, \mathfrak{N}\subseteq\mathfrak{A}_j\subseteq\mathfrak{M}$ be successively independent over $\mathfrak{N}$ and $x_j\in\mathfrak{A}_j$ be self-adjoint such that $\mathcal{E}_{\mathfrak{N}}(x_j)=0$. Then there exists a constant $C(p, r)$ such that
\begin{eqnarray*}
\left\|\sum_{j=1}^nx_j\right\|_p^p\leq C(p, r)\left\{\sum_{j=1}^n\|x_j\|_p^p+\left(\sum_{j=1}^n\|x_j^r\|\right)^{\frac{p}{r}}\right\}.
\end{eqnarray*}
\end{theorem}
\begin{proof}
We use the noncommutative Bennett inequality, but for this end, we replace $\Phi(\alpha)$ by $\alpha\log(1+\alpha)-\alpha$, which is clearly smaller
than $\Phi(\alpha)$ for any $\alpha\geq0$.

Let us fix a number $s\geq0$ and consider $y_j=x_j\chi_{(-\infty,s]}(x_j)\in \mathfrak{A}_j$. It follows from $y_j\leq x_j$ and the positivity of $\mathcal{E}_{\mathfrak{N}}$ that $\mathcal{E}_{\mathfrak{N}}(y_j)\leq\mathcal{E}_{\mathfrak{N}}(x_j)=0$. In addition, $\mathcal{E}_{\mathfrak{N}}$ is norm decreasing, so
\begin{eqnarray*}
\sum_{j=1}^n\mathcal{E}_{\mathfrak{N}}(|x_j|^r)\leq \sum_{j=1}^n\|x_j^r\|:=B\quad {\rm and} \quad b:=\sum_{j=1}^n\|y_j^r\|\geq
\sum_{j=1}^n\mathcal{E}_{\mathfrak{N}}(|y_j|^r).
\end{eqnarray*}
Further, $b \leq B$ since $\|y_j\|\leq \|x_j\|\leq M$, where $M:=\max_{1\leq j\leq n}\|x_j\|$.
It follows from the noncommutative Bennett inequality \eqref{MOS2} that
\begin{align}\label{T}
\tau\left(\chi_{[t,\infty)}\sum_{j=1}^ny_j\right)&\leq \exp\left\{\frac{-b}{M^r}\Phi\left(\frac{tM^{r-1}}{b}\right)\right\}\\\nonumber
&\leq \exp\left\{\frac{-b}{M^r}\left[\frac{tM^{r-1}}{b}\log\left(1+\frac{tM^{r-1}}{b}\right)-\frac{tM^{r-1}}{b}\right]\right\}\\\nonumber
&\leq\exp\left\{-\frac{t}{M}\left(\log\left(1+\frac{tM^{r-1}}{B}\right)-1\right)\right\}
\end{align}
for all $t>0$. We have
\begin{eqnarray}\label{P}
\chi_{[t,\infty)}\left(\sum_{j=1}^nx_j\right)\prec \chi_{[t,\infty)}\left(\sum_{j=1}^ny_j\right)\vee\left(\vee_{j=1}^n\chi_{[s,\infty)}(x_j)\right)
\end{eqnarray}
for all $t>0$. To show this we have to prove that
\begin{eqnarray*}
\chi_{[t,\infty)}\left(\sum_{j=1}^nx_j\right)\wedge\left(\chi_{[0,t)}\left(\sum_{j=1}^ny_j\right)\wedge
\left(\wedge_{j=1}^n\chi_{(-\infty,s)}(x_j)\right)\right)=0.
\end{eqnarray*}
Let $\xi\in\chi_{[t,\infty)}\left(\sum_{j=1}^nx_j\right)(\mathfrak{H})\bigcap\left(\chi_{[0,t)}\left(\sum_{j=1}^ny_j\right)(\mathfrak{H})\bigcap
\left(\cap_{j=1}^n\chi_{(-\infty,s)}(x_j)(\mathfrak{H})\right)\right)$. Then
$$\left\langle\left(\sum_{j=1}^nx_j\right)\xi,\xi\right\rangle\geq t$$ and
\begin{eqnarray*}
t>\left\langle\left(\sum_{j=1}^ny_j\right)\xi,\xi\right\rangle&=&\left\langle\left(\sum_{j=1}^nx_j\chi_{(-\infty,s)}(x_j)\right)\xi,\xi\right\rangle=\left\langle\left(\sum_{j=1}^nx_j\right)\xi,\xi\right\rangle\geq t
\end{eqnarray*}
since $\xi\in\cap_{j=1}^n\chi_{(-\infty,s)}(x_j)(\mathfrak{H})$.
Therefore
\begin{eqnarray*}
\chi_{[t,\infty)}\left(\sum_{j=1}^nx_j\right)\wedge\left(\chi_{[0,t)}\left(\sum_{j=1}^ny_j\right)\wedge
\left(\wedge_{j=1}^n\chi_{(-\infty,s)}(x_j)\right)\right)=0.
\end{eqnarray*}
We deduce from Lemma \ref{NE}(ii) that
\begin{eqnarray*}
\chi_{[t,\infty)}\left(\sum_{j=1}^nx_j\right)\prec\left(\chi_{[0,t)}\left(\sum_{j=1}^ny_j\right)\wedge
\left(\wedge_{j=1}^n\chi_{(-\infty,s)}(x_j)\right)\right)^{\perp}.
\end{eqnarray*}
and this gives us inequality \eqref{P}.\\
Using \eqref{P}, we get
\begin{eqnarray*}
\tau\left(\chi_{[t,\infty)}\left(\sum_{j=1}^nx_j\right)\right)
&\leq& \tau\left(\chi_{[t,\infty)}\left(\sum_{j=1}^ny_j\right)\right)+\tau\left(\vee_{j=1}^n\chi_{[s,\infty)}(x_j)\right)\\
&&\qquad\qquad\qquad\qquad ({\rm by~ Lemma~ \ref{NE}(ii-iii)})\\
&\leq&\exp\left\{-\frac{t}{M}\left(\log\left(1+\frac{tM^{r-1}}{B}\right)-1\right)\right\}\\
&&\quad+\sum_{j=1}^n\tau\left(\chi_{[s,\infty)}(x_j)\right)\\
&&\qquad\qquad\qquad \quad ({\rm by~ Lemma~ \ref{NE}(iv)}, ~\eqref{T})
\end{eqnarray*}
for any $t>0$. An easy investigation shows that the latter inequality holds if we replace $M$ by any number $L$ with $L\geq M$. In addition, it holds for all $0<L \leq M$ since the function
 $$f(\alpha)=\exp\left\{-\frac{t}{\alpha}\left(\log\left(1+\frac{t\alpha^{r-1}}{B}\right)-1\right)\right\}$$
is decreasing for any $\alpha>0$. Thus
 \begin{align}\label{IN1}
\tau\left(\chi_{[t,\infty)}\left(\sum_{j=1}^nx_j\right)\right)
\leq \exp\left\{-\frac{t}{L}\left(\log\left(1+\frac{tL^{r-1}}{B}\right)-1\right)\right\}+\sum_{j=1}^n\tau\left(\chi_{[s,\infty)}(x_j)\right)
\end{align}
for all $t>0, s>0, L>0$.

Next, we deal with the modulus of $\sum_{j=1}^nx_j$. Inequality \eqref{IN1} together with \eqref{tc} imply that
\begin{eqnarray*}
&&\hspace{-1cm}\tau\left(\chi_{[t,\infty)}\left(\left|\sum_{j=1}^nx_j\right|\right)\right)\\
&\leq&2\exp\left\{-\frac{t}{L}\left(\log\left(1+\frac{tL^{r-1}}{B}\right)-1\right)\right\}+ \sum_{j=1}^n\tau\left(\chi_{[s,\infty)}(|x_j|)\right)\,.
\end{eqnarray*}
Now, by putting first $L=s$ and $s=\frac{t}{\gamma}$, where $\gamma>0$ we obtain that
\begin{eqnarray*}
&&\hspace{-1cm}\tau\left(\chi_{[t,\infty)}\left(\left|\sum_{j=1}^nx_j\right|\right)\right)\\
&\leq&
2\exp\left\{-\gamma\left(\log\left(1+\frac{t^r}{\gamma^{r-1} B}\right)-1\right)\right\}+ \sum_{j=1}^n\tau\left(\chi_{[t,\infty)}\left(\gamma|x_j|\right)\right).
\end{eqnarray*}
For $p\geq2$, we have
\begin{align}\label{IN2}
&\hspace{-1cm}\left\|\sum_{j=1}^nx_j\right\|_p^p\nonumber\\
&=\int_0^{\infty}pt^{p-1}\tau\left(\chi_{[t,\infty)}\left(\left|\sum_{j=1}^nx_j\right|\right)\right)dt\nonumber\\
&\leq \sum_{j=1}^n\int_0^{\infty}pt^{p-1}\tau\left(\chi_{[t,\infty)}(\gamma\left|x_j\right|)\right)dt\nonumber\\
&\quad+
2p\int_0^{\infty}t^{p-1}\exp\left\{-\gamma\left(\log\left(1+\frac{t^r}{\gamma^{r-1}B}\right)-1\right)\right\}dt\nonumber\\
&\leq \gamma^p\sum_{j=1}^n\|x_j\|_p^p+\frac{2p}{r}(\gamma^{r-1} B)^{\frac{p}{r}}e^{\gamma}\int_0^{\infty}\beta^{\frac{p-r}{r}}(1+\beta)^{-\gamma}d\beta\,,
\end{align}
where we use the change of variable $\beta=\frac{t^r}{\gamma^{r-1} B}$.
Next let $\gamma$ be such that the last integral of the above inequality is convergent, i.e. let us choose $\gamma>\frac{p}{r}$. With this choice
of $\gamma$ inequality \eqref{IN2} implies the Rosenthal inequality with
\begin{eqnarray*}
C(p, r)=\max\left\{ \gamma^p, \frac{2p}{r}\gamma^{\frac{p(r-1)}{r}}e^{\gamma}\int_0^{\infty}\beta^{\frac{p-r}{r}}(1+\beta)^{-\gamma}d\beta\right\}.
\end{eqnarray*}
\end{proof}
In the commutative setting, we have
\begin{corollary}
Let $1\leq r\leq2\leq p<\infty$ and $X_1, \cdots, X_n$ be independent real random variables with expected values $\mathbb{E}(X_j)=0\,\,(j=1, \cdots, n)$. Then
\begin{eqnarray*}
\mathbb{E}\left(\left|\sum_{j=1}^nX_j\right|^p\right)\leq C(p)\left\{\sum_{j=1}^n\mathbb{E}\left(|X_j|^p\right)+\left(\sum_{j=1}^n\mathbb{E}\left(|X_j|^r\right)\right)^{p/r}\right\},
\end{eqnarray*}
where $C(p)=\displaystyle{\min_{\gamma>p/2}}\max\left\{ \gamma^p, \frac{2p}{r}\gamma^{\frac{p(r-1)}{r}}e^{\gamma}\int_0^{\infty}\beta^{\frac{p-r}{r}}(1+\beta)^{-\gamma}d\beta\right\}$.
\end{corollary}
\bigskip

\textbf{Acknowledgment.} We thank S. Dirksen for pointing out the reference \cite{DIR} to our attention.


\bigskip

\begin{thebibliography}{99}

\bibitem{BEN} G. Bennett, \textit{Probability inequalities for the sum of independent random variables}, J. Amer. Statist. Assoc. \textbf{57} (1962), 33-–45.

\bibitem{1} V. Bentkus, \textit{On Hoeffding's inequalities}, Ann. Probab. \textbf{32} (2004), no. 2, 1650–-1673.

 \bibitem{BUR} D. L. Burkholder, \textit{Distribution function inequalities for martingales}, Ann. Probab. \textbf{1} (1973), 19--42.

\bibitem{CO} J. B. Conway, \textit{ A Course in Functional Analysis}, Second, Graduate Texts in Mathematics,
springer--Verlag, New York, 1990.


\bibitem{DIR} S. Dirksen, B. de Pagter, D. Potapov and F. Sukochev, \textit{Rosenthal inequalities in noncommutative symmetric spaces}, J. Funct. Anal. \textbf{261} (2011), 2890--2925.

\bibitem{HOE} W. Hoeffding, \textit{Probability inequalities for sums of bounded random variables}, J. Amer. Statist. Assoc. \textbf{58} (1963), 13-–30.

\bibitem{HU} X.-L. Hu, \textit{An extension of Rosenthal's inequality}, Appl. Math. Comput. \textbf{218} (2011), no. 8, 4638--4640.

\bibitem{JSZ} W. B. Johnson, G. Schechtman and J. Zinn, \textit{Best constants in moment inequalities for linear combinations of independent and exchangeable random variables}, Ann. Probab. \textbf{13} (1985), no. 1, 234--253.

\bibitem {JZ} M. Junge and Q. Zeng,\textit{ Noncommutative Bennett and Rosenthal inequalities},
Ann. Probab. \textbf{41} (2013), no. 6, 4287--4316.

\bibitem{JX} M. Junge and Q. Xu, \textit{Noncommutative Burkholder/Rosenthal inequalities. II. Applications}, Israel J. Math. \textbf{167} (2008), 227--282.

\bibitem{MTS} M. S. Moslehian, M. Tominaga and K.-S. Saito, \textit{Schatten $p$-norm inequalities related to an extended operator parallelogram law}, Linear Algebra Appl. \textbf{435} (2011), no. 4, 823--829.

\bibitem{N} E. Nelson, \textit{ Note on noncommutative integration}, J. Funct. Anal. \textbf{15} (1974), 103--116.

\bibitem{P} B. de Pagter, \textit{ Noncommutative Banach function spaces}, Positivity, 197-227, Trends Math., Birkh\"user, Basel, (2007).

\bibitem{ROS} H. P. Rosenthal, \textit{On the subspaces of $L_p (p > 2)$ spanned by sequences of independent random variables}, Israel J. Math. \textbf{8} (1970), 273–-303.

\bibitem{RUS} M. B. Ruskai, \textit{Inequalities for traces on von Neumann algebras}, Comm. Math. Phys. \textbf{26} (1972), 280--289.

\bibitem{SAD} Gh. Sadeghi, \textit{Non-commutative Orlicz spaces associated to a modular on $\tau$-measurable operators}, J. Math. Anal. Appl.
\textbf{395} (2012), no. 2, 705--715.

\bibitem{3} J. A. Tropp, \textit{User-friendly tail bounds for sums of random matrices},
Found. Comput. Math. \textbf{12} (2012), no. 4, 389–-434.

\bibitem{Y} F. J. Yeadon, \textit{ Noncommutative $L^p$-spaces}, Proc. Camb. Phil. Soc. \textbf{77} (1975), 91--102.

\bibitem{Xu1} Q. Xu, \textit{Embedding of $C_q$ and $R_q$ into noncommutative $L_p$-spaces, $1\leq p<q \leq 2$}, Math. Ann. \textbf{335} (2006), no. 1, 109–-131.

\bibitem{Xu2} Q. Xu, \textit{Operator spaces and noncommutative $L_p$}, Lectures in the Summer School on Banach spaces and Operator spaces, Nankai University China, 2007.

\end{thebibliography}
\end{document}